\documentclass{amsart}
\usepackage{amsmath,amssymb,latexsym,stackrel,mathtools}
\usepackage[margin=1in]{geometry}

\numberwithin{equation}{section}

\newtheorem{theorem}{Theorem}[section]
\newtheorem{proposition}[theorem]{Proposition}

\newtheorem{corollary}[theorem]{Corollary}

\theoremstyle{definition}
\newtheorem{remark}[theorem]{Remark}
\newtheorem{example}[theorem]{Example}
\newtheorem{definition}[theorem]{Definition}

\newcommand{\PP}{\mathbb{P}}
\newcommand{\QQ}{\mathbb{Q}}
\newcommand{\TT}{\mathbb{T}}
\newcommand{\ZZ}{\mathbb{Z}}

\newcommand{\bfb}{\mathbf{b}}
\newcommand{\bfc}{\mathbf{c}}
\newcommand{\bfd}{\mathbf{d}}
\newcommand{\bfe}{\mathbf{e}}
\newcommand{\bfg}{\mathbf{g}}
\newcommand{\bfp}{\mathbf{p}}

\newcommand{\bfx}{\mathbf{x}}
\newcommand{\Lbfx}{{}^L\mathbf{x}}
\newcommand{\Rbfx}{{}^R\mathbf{x}}
\newcommand{\bfy}{\mathbf{y}}
\newcommand{\Lbfy}{{}^L\bfy}

\newcommand{\Rbfy}{{}^R\bfy}

\newcommand{\bfz}{\mathbf{z}}
\newcommand{\bfZ}{\mathbf{Z}}

\newcommand{\LB}{{}^L\!B}
\newcommand{\RB}{{}^R\!B}

\newcommand{\cA}{\mathcal{A}}
\newcommand{\LcA}{{}^L\!\!\mathcal{A}}
\newcommand{\RcA}{{}^R\!\!\mathcal{A}}
\newcommand{\cF}{\mathcal{F}}

\newcommand{\bin}{{\operatorname{bin}}}

\renewcommand{\sf}{{\operatorname{sf}}}
\newcommand{\Trop}{\operatorname{Trop}}

\newcommand{\erase}[1]{{}}

\title{Companion cluster algebras to a generalized cluster algebra}
\author{Tomoki Nakanishi}
\author{Dylan Rupel}

\begin{document}
\begin{abstract}
 We study the $c$-vectors, $g$-vectors, and $F$-polynomials for generalized cluster algebras satisfying a normalization condition and a power condition recovering classical recursions and separation of additions formulas.  We establish a relationship between the $c$-vectors, $g$-vectors, and $F$-polynomials of such a generalized cluster algebra and its (left- and right-) companion cluster algebras.  Our main result states that the cluster variables and coefficients of the (left- and right-) companion cluster algebras can be recovered via a specialization of the $F$-polynomials.
\end{abstract}
\maketitle

\section{Introduction}
Cluster algebras have risen to prominence as the correct algebraic/combinatorial language for describing a certain class of recursive calculations.  These recursions appear in many forms across various disciplines including Poisson geometry \cite{gekhtman-shapiro-vainshtein}, combinatorics \cite{musiker-propp}, hyperbolic geometry \cite{fock-goncharov,fomin-shapiro-thurston,musiker-schiffler-williams}, representation theory of associative algebras \cite{caldero-chapoton,caldero-keller,bmrrt,rupel1,qin,rupel2}, mathematical physics \cite{eager-franco}, and quantum groups \cite{kimura,geiss-leclerc-schroer,kimura-qin,berenstein-rupel}.  In the current standard theory a product of cluster variables, one known and one unknown, is equal to a binomial in other known quantities.  Recently examples have emerged in the context of hyperbolic orbifolds \cite{chekhov-shapiro}, exact WKB analysis \cite{iwaki-nakanishi}, and quantum groups \cite{gleitz,berenstein-greenstein-rupel} that require a more general setup: these \emph{binomial} exchange relations should be replaced by \emph{polynomial} exchange relations.

The general study of such \emph{generalized} cluster algebras was initiated by Chekhov and Shapiro \cite{chekhov-shapiro} where an analogue of the classical Laurent Phenomenon was established.  Following these developments, the first author \cite{nakanishi} studied the analogues of $c$-vectors, $g$-vectors, and $F$-polynomials for a class of generalized cluster algebras satisfying a \emph{normalization condition} and a \emph{reciprocity condition}.  In that work, relationships between these $c$- and $g$-vectors with the corresponding quantities for certain \emph{companion} cluster algebras were established.  Our goal in the present paper is to extend these results to the case when the reciprocity condition is replaced by a weaker \emph{power condition} and to clarify the corresponding relationships between $F$-polynomials, $x$-variables, and $y$-variables.  The main message of this note, continuing from \cite{nakanishi}, is as follows: the generalized cluster algebras are as good and natural as ordinary cluster algebras.  Also in this direction, analogues of the classical greedy bases from \cite{lee-li-zelevinsky} have been constructed for rank 2 generalized cluster algebras by the second author \cite{rupel3}.

In order to state our main theorem we will need to fix some notation.  A cluster algebra $\cA(\bfx,\bfy,B)\subset\cF$ is defined recursively from the initial data of a seed $(\bfx,\bfy,B)$ where $\bfy=(y_1,\ldots,y_n)$ is a collection of elements from a semifield $\PP$, $\bfx=(x_1,\ldots,x_n)$ is a collection of algebraically independent elements in a degree $n$ purely transcendental extension $\cF$ of $\QQ\PP$ (in particular, we may identify $\cF$ with the rational function field $\QQ\PP(\bfx)$) where $\QQ\PP$ is the field of fractions of the group ring $\ZZ\PP$, and $B=(b_{ij})$ is a skew-symmetrizable $n\times n$ matrix.  A generalized cluster algebra $\cA=\cA(\bfx,\bfy,B,\bfZ)\subset\cF$ requires the additional data of a collection of \emph{exchange polynomials} $\bfZ=(Z_1,\ldots,Z_n)$ where
\[Z_i(u)=z_{i,0}+z_{i,1}u+\cdots+z_{i,d_i-1}u^{d_i-1}+z_{i,d_i}u^{d_i}\]
with each $z_{i,s}\in\PP$ and $z_{i,0}=z_{i,d_i}=1$.  Write $\bfz=(z_{i,s})$ ($1\le i\le n$, $0\le s\le d_i$).

Write $D=(d_i\delta_{ij})$ for the diagonal $n\times n$ matrix. Denote by $\bfx^{1/\bfd}$ the collection $(x_1^{1/d_1},\ldots,x_n^{1/d_n})$ in the extension field $\QQ\PP(\bfx^{1/\bfd})$ of $\cF$.  Define the \emph{left-companion cluster algebra} $\LcA$ of $\cA$ to be the cluster algebra $\cA(\bfx^{1/\bfd},\bfy,DB)\subset\QQ\PP(\bfx^{1/\bfd})$.  Write $(\Lbfx^t,\Lbfy^t,\LB^t)$ for the seed associated to vertex $t\in\TT_n$ in the construction of $\LcA$ and denote by ${}^L\!\bfc^t_j$, ${}^L\!\bfg^t_j$, and ${}^L\!F^t_j$ the $c$-vectors, $g$-vectors, and $F$-polynomials of $\LcA$.

Let $\bfz^\bin=(z_{i,s}^\bin)$ where $z_{i,s}^\bin={d_i\choose s}$.  Then we write $x_i^t\big|_{\bfz=\bfz^\bin}\in\cF$ and $y_j^t\big|_{\bfz=\bfz^\bin}\in\PP$ for the variables obtained by applying equations \eqref{th:GCA x-variables} and \eqref{th:GCA y-variables} respectively using the specialized $F$-polynomials $F^t_j(\bfy,\bfz^{\bin})$ in place of the generic $F$-polynomials $F^t_j(\bfy,\bfz)$.  Our first main result is the following.
\begin{theorem}\label{th:left companions}
 We have $x_i^t\big|_{\bfz=\bfz^\bin}=\big({}^L\!x^t_i\big)^{d_i}$ and $y_j^t\big|_{\bfz=\bfz^\bin}={}^L\!y^t_i$.
\end{theorem}

Denote by $\bfy^\bfd$ for the collection $(y_1^{d_1},\ldots,y_n^{d_n})$ in $\PP$.  Define the \emph{right-companion cluster algebra} $\RcA$ of $\cA$ to be the cluster algebra $\cA(\bfx,\bfy^\bfd,BD)\subset\QQ\PP(\bfx)$.  Write $(\Rbfx^t,\Rbfy^t,\RB^t)$ for the seed associated to vertex $t\in\TT_n$ in the construction of $\RcA$ (see Section~\ref{sec:CA} for details).

Write $x_i^t\big|_{\bfz=\boldsymbol{0}}\in\cF$ and $y_j^t\big|_{\bfz=\boldsymbol{0}}\in\PP$ for the variables obtained by applying equations \eqref{th:GCA x-variables} and \eqref{th:GCA y-variables} respectively using the specialized $F$-polynomials $F^t_j(\bfy,\boldsymbol{0})$ in place of the generic $F$-polynomials $F^t_j(\bfy,\bfz)$.  Our second main result is the following.
\begin{theorem}\label{th:right companions}
 We have $x_i^t\big|_{\bfz=\boldsymbol{0}}={}^Rx^t_i$ and $\big(y_j^t\big|_{\bfz=\boldsymbol{0}}\big)^{d_j}={}^Ry^t_j$.
\end{theorem}

\section{Cluster Algebras}\label{sec:CA}
A \emph{semifield} is a multiplicative abelian group $(\PP,\cdot)$ together with an auxiliary addition $\oplus:\PP\times\PP\to\PP$ which is associative, commutative and satisfies the usual distributivity with the multiplication of $\PP$.  Write $\ZZ\PP$ for the group ring of $\PP$.  Since $\PP$ is necessarily torsion-free (see e.g. \cite[Sec. 5]{fomin-zelevinsky1}), $\ZZ\PP$ is a domain \cite[Sec. 2]{fomin-zelevinsky1} and we write $\QQ\PP$ for its field of fractions.  There are two main examples of semifields that will be most relevant for our purposes.

\begin{example}\mbox{}
 \begin{enumerate}
  \item The \emph{universal semifield} $\QQ_\sf(y_1,\ldots,y_n)$ is the set of rational functions in the variables $y_1,\ldots,y_n$ which can be written in a subtraction-free form.  Addition and multiplication in the universal semifield are the ordinary operations on rational functions.  The semifield $\QQ_\sf(y_1,\ldots,y_n)$ is universal in the following sense.  Each element of $\QQ_\sf(y_1,\ldots,y_n)$ can be written as a ratio of positive polynomials in $\ZZ_{\ge0}[y_1,\ldots,y_n]$ so that for any other semifield $\PP$ there is a specialization homomorphism $\QQ_\sf(y_1,\ldots,y_n)\to\PP$, given by $y_i\mapsto p_i$ and $1\mapsto1$, which respects the semifield structure for any choice of $p_1,\ldots,p_n\in\PP$.
  \item The \emph{tropical semifield} $\Trop(y_1,\ldots,y_n)$ is the free (multiplicative) abelian group generated by $y_1,\ldots,y_n$ with auxiliary addition $\oplus$ defined by 
  \[\stackrel[j=1]{n}{\prod} y_j^{a_j}\oplus\stackrel[j=1]{n}{\prod} y_j^{b_j}=\stackrel[j=1]{n}{\prod} y_j^{\min(a_j,b_j)}.\]
  The group ring of $\PP=\Trop(y_1,\ldots,y_n)$ is the Laurent polynomial ring $\ZZ[y_1^{\pm1},\ldots,y_n^{\pm1}]$ while $\QQ\PP=\QQ(y_1,\ldots,y_n)$.
 \end{enumerate}
\end{example}
Fix a semifield $\PP$ and write $\cF=\QQ\PP(w_1,\ldots,w_n)$ for the field of rational functions in algebraically independent variables $w_1,\ldots,w_n$.  A \emph{(labeled) seed $(\bfx,\bfy,B)$ over $\PP$} consists of the following data:
\begin{itemize}
 \item an algebraically independent collection $\bfx=(x_1,\ldots,x_n)$, called a \emph{cluster}, consisting of elements from $\cF$ called \emph{cluster variables} or \emph{$x$-variables};
 \item a collection $\bfy=(y_1,\ldots,y_n)$ of elements from $\PP$ called \emph{coefficients} or \emph{$y$-variables};
 \item an $n\times n$ skew-symmetrizable matrix $B=(b_{ij})$ called the \emph{exchange matrix}.
\end{itemize}

The main ingredient in the definition of a cluster algebra is the notion of mutation for seeds.  For notational convenience we abbreviate $[b]_+=\max(b,0)$.
\begin{definition}\label{def:CA mutation}
 For $1\le k\le n$ we define the \emph{seed mutation in direction $k$} by $\mu_k(\bfx,\bfy,B)=(\bfx',\bfy',B')$ where
 \begin{itemize}
  \item the cluster $\bfx'=(x'_1,\ldots,x'_n)$ is given by $x'_i=x_i$ for $i\ne k$ and $x'_k$ is determined using the \emph{exchange relation}:
  \begin{equation}\label{eq:CA exchange}
   x'_kx_k=\bigg(\prod_{i=1}^n x_i^{[-b_{ik}]_+}\bigg)\frac{1+\hat y_k}{1\oplus y_k},\quad \hat y_k=y_k\prod_{i=1}^n x_i^{b_{ik}};
  \end{equation}
  \item the coefficient tuple $\bfy'=(y'_1,\ldots,y'_n)$ is given by $y'_k=y_k^{-1}$ and for $j\ne k$ we set
  \begin{equation}\label{eq:y mutation}
   y'_j=y_j y_k^{[b_{kj}]_+}\big(1\oplus y_k\big)^{-b_{kj}};
  \end{equation}
  \item the matrix $B'=(b'_{ij})$ is given by
  \begin{equation}\label{eq:CA matrix mutation}
   b'_{ij}=\begin{cases}
           -b_{ij} & \text{ if $i=k$ or $j=k$;}\\
           b_{ij}+[b_{ik}]_+b_{kj}+b_{ik}[-b_{kj}]_+ & \text{ otherwise.}\\
          \end{cases}
  \end{equation}
 \end{itemize}
\end{definition}

Write $\TT_n$ for the $n$-regular tree with edges labeled by the set $\{1,\ldots,n\}$ so that the $n$ edges emanating from each vertex receive different labels.  We write $t\stackrel{k}{\text{---}}t'$ to denote two vertices $t$ and $t'$ of $\TT_n$ connected by an edge labeled by $k$.  A \emph{cluster pattern $\Sigma$ over $\PP$} is an assignment of a seed $\Sigma^t$ to each vertex $t\in\TT_n$ such that whenever $t\stackrel{k}{\text{---}}t'$ we have $\mu_k\Sigma^t=\Sigma^{t'}$, that is $\Sigma^t$ and $\Sigma^{t'}$ are related by the seed mutation in direction $k$ whenever $t$ and $t'$ are adjoined by an edge labeled by $k$.  Fix a choice of initial vertex $t_0$, we will write $\Sigma^{t_0}=(\bfx,\bfy,B)$ while for an arbitrary vertex $t\in\TT_n$ we write $\Sigma^t=(\bfx^t,\bfy^t,B^t)$ where
\[\bfx^t=(x^t_1,\ldots,x^t_n),\quad\quad \bfy^t=(y^t_1,\ldots,y^t_n),\quad\quad B^t=(b^t_{ij}).\]
Note that every seed $\Sigma^t$ for $t\in\TT_n$ is uniquely determined once we have specified $\Sigma^{t_0}$.  Moreover, it is important to note that the exchange matrices $B^t$ are independent of the initial choice of $\bfx$ and $\bfy$.
\begin{definition}
 The \emph{cluster algebra} $\cA(\bfx,\bfy,B)$ is the $\ZZ\PP$-subalgebra of $\cF$ generated by all cluster variables from seeds appearing in the cluster pattern $\Sigma$, more precisely 
 \[\cA(\bfx,\bfy,B)=\ZZ\PP[x^t_i:t\in\TT_n,1\le i\le n]\subset\cF.\]
\end{definition}

A priori the most one can say about these constructions is that the cluster variables $x^t_i$ admit a description as subtraction-free rational expressions in the cluster variables of $\bfx$ with coefficients in $\ZZ\PP$ and that the coefficients $y^t_j$ admit a description as subtraction-free rational expressions in $\QQ_\sf(\bfy)$.  More precisely, to see this claim for $x^t_i$ we may, for each initial seed $(\bfx,\bfy,B)$,
\begin{itemize}
 \item replace the $x$- and $y$-variables by formal indeterminates (which by abuse of notation we denote by the same symbols);
 \item replace the semifield $\PP$ by the tropical semifield $\Trop(\bfy)$;
 \item replace $\cF$ by $\QQ(\bfx,\bfy)$ and opt to perform all calculations here.
\end{itemize}
Since no subtraction occurs in the recursions \eqref{eq:CA exchange}, we obtain in this way \emph{$X$-functions} $X^t_i\in\QQ_\sf(\bfx,\bfy)$.  Alternatively performing the $y$-mutations \eqref{eq:y mutation} inside $\QQ_\sf(\bfy)$ we obtain \emph{$Y$-functions} $Y^t_j\in\QQ_\sf(\bfy)$.  By the universality of the semifield $\QQ_\sf(\bfy)$ we may recover the original coefficient $y^t_j$ by the specialization $Y^t_j\big|_\PP$.  Taking this specialization where $\PP=\Trop(\bfy)$ we obtain monomials $Y^t_j\big|_{\Trop(\bfy)}=\stackrel[i=1]{n}{\prod} y_i^{c^t_{ij}}$ where we write $C^t$ for the resulting matrix whose columns $\bfc^t_j\in\ZZ^n$ are called \emph{$c$-vectors}.  Note that the $c$-vectors only depend on the initial exchange matrix $B$ and not on the choice of initial cluster $\bfx$.
\begin{proposition}\cite[Eq. 5.9]{fomin-zelevinsky4}\label{prop:CA c-vectors}
 The $c$-vectors satisfy the following recurrence relation for $t\stackrel{k}{\text{---}}t'$:
 \begin{equation}\label{eq:CA c-vector recursion}
  c^{t'}_{ij}=\begin{cases}-c^t_{ik} & \text{if $j=k$;}\\ c^t_{ij}+c^t_{ik}[b^t_{kj}]_++[-c^t_{ik}]_+b^t_{kj} & \text{if $j\ne k$.}\end{cases}
 \end{equation}
\end{proposition}

Obtaining the cluster variable $x^t_i$ from $X^t_i$ is more interesting and will be discussed further below.  As a first step toward this goal, we note that the cluster algebra $\cA$ admits the following remarkable ``Laurent Phenomenon''.  
\begin{theorem}\cite[Th. 3.1]{fomin-zelevinsky1}\label{th:CA laurent}
 Fix an initial seed $(\bfx,\bfy,B)$ over a semifield $\PP$.  For any vertex $t\in\TT_n$ each cluster variable $x^t_i$ can be expressed as a Laurent polynomial in $\bfx$ with coefficients in $\ZZ\PP$.
\end{theorem}

For a seed $(\bfx,\bfy,B)$ over $\PP=\Trop(\bfy)$ we may apply Theorem~\ref{th:CA laurent} to write each $X$-function as an element of $\ZZ[\bfx^{\pm1},\bfy^{\pm1}]$.  Moreover, $y$-variables never appear in the denominators of the $X$-functions.
\begin{proposition}\cite[Prop. 3.6]{fomin-zelevinsky4}\label{prop:CA polynomial coefficients}
 Each $X$-function $X^t_i$ is contained in $\ZZ[\bfx^{\pm1},\bfy]$.
\end{proposition}
In fact, the $X$-functions are homogeneous with repect to a certain $\ZZ^n$-grading on $\ZZ[\bfx^{\pm1},\bfy]$.  Write $\bfb_j\in\ZZ^n$ for the $j^{th}$ column of $B$.
\begin{proposition}\cite[Prop. 6.1, Prop. 6.6]{fomin-zelevinsky4}\label{prop:CA g-vectors}
 Under the $\ZZ^n$-grading 
 \[\deg(x_i)=\bfe_i\quad\quad\text{and}\quad\quad \deg(y_j)=-\bfb_j,\]
 each $X$-function is homogeneous and we write $\deg\big(X^t_j\big)=\bfg^t_j=\sum\limits_{i=1}^n g^t_{ij}\bfe_i$.  Moreover, these \emph{$g$-vectors} satisfy the following recurrence relation for $t\stackrel{k}{\text{---}}t'$:
 \begin{equation}\label{eq:CA g-vector recursion}
  g^{t'}_{ij}=\begin{cases}g^t_{ij} & \text{if $j\ne k$;}\\ -g^t_{ik}+\sum\limits_{\ell=1}^n g_{i\ell}^t[-b_{\ell k}^t]_+-\sum\limits_{\ell=1}^n b_{i\ell}^t[-c_{\ell k}^t]_+ & \text{if $j=k$.}\end{cases}
 \end{equation}
\end{proposition}

Following Proposition~\ref{prop:CA polynomial coefficients} we may define \emph{$F$-polynomials} $F^t_i(\bfy)\in\ZZ[\bfy]$ via the specialization $F^t_i(\bfy)=X^t_i(\boldsymbol{1},\bfy)$, i.e. by setting all initial cluster variables $x_j$ to 1.  The $F$-polynomials satisfy a recurrence relation analogous to \eqref{eq:CA exchange}.
\begin{proposition}\cite[Prop. 5.1]{fomin-zelevinsky4}\label{prop:CA F-polynomials}
 The $F$-polynomials satisfy the following recurrence relation for $t\stackrel{k}{\text{---}}t'$:
 \begin{equation}\label{eq:CA F-polynomial recursion}
  F_j^{t'}=\begin{cases}F_j^t & \text{if $j\ne k$;}\\ \displaystyle\big(F_k^t\big)^{-1}\bigg(\stackrel[i=1]{n}{\prod} y_i^{[-c_{ik}^t]_+}\big(F_i^t\big)^{[-b_{ik}^t]_+}\bigg)\bigg(1+\stackrel[i=1]{n}{\prod} y_i^{c_{ik}^t}\big(F_i^t\big)^{b_{ik}^t}\bigg) & \text{if $j=k$.}\end{cases}
 \end{equation}
\end{proposition}
Notice that each $F$-polynomial admits an expression as a subtraction-free rational expression and thus may be considered as an element of $\QQ_\sf(\bfy)$, in particular the specialization $F^t_i\big|_\PP$ makes sense for any semifield $\PP$.  With this we may obtain a description of the $y$-variables in terms of the $c$-vectors and the specializations of the $F$-polynomials.
\begin{theorem}\cite[Prop. 3.13]{fomin-zelevinsky4}\label{th:CA y-variables}
 Fix an initial seed $(\bfx,\bfy,B)$ over a semifield $\PP$.  For any vertex $t\in\TT_n$ each coefficient $y^t_j$ of $\cA(\bfx,\bfy,B)$ can be computed as
 \[y^t_j=\bigg(\stackrel[i=1]{n}{\prod} y_i^{c^t_{ij}}\bigg)\stackrel[i=1]{n}{\prod}F^t_i\big|_\PP(\bfy)^{b^t_{ij}} .\]
\end{theorem}
Finally, we obtain a ``separation of additions'' formula for the cluster variables $x^t_i$ in terms of the $g$-vectors and the $F$-polynomials.
\begin{theorem}\cite[Cor. 6.3]{fomin-zelevinsky4}\label{th:CA x-variables}
 Fix an initial seed $(\bfx,\bfy,B)$ over a semifield $\PP$.  For any vertex $t\in\TT_n$ each cluster variable $x^t_j$ of $\cA(\bfx,\bfy,B)$ can be computed as
 \[x^t_j=\bigg(\stackrel[i=1]{n}{\prod} x_i^{g^t_{ij}}\bigg)\frac{F^t_j\big|_\cF(\hat\bfy)}{F^t_j\big|_\PP(\bfy)}.\]
\end{theorem}

\section{Generalized Cluster Algebras}\label{sec:GCA}
Let $(\bfx,\bfy,B)$ be a seed over the semifield $\PP$.  Fix a collection $\bfZ=(Z_1,\ldots,Z_n)$ of positive degree \emph{exchange polynomials} 
\[Z_i(u)=z_{i,0}+z_{i,1}u+\cdots+z_{i,d_i-1}u^{d_i-1}+z_{i,d_i}u^{d_i}\in\ZZ\PP[u]\]
such that $z_{i,s}\in\PP$ for $0\le s\le d_i$ and $z_{i,0}=z_{i,d_i}=1$.  It will often be convenient to write $\bfz=(z_{i,s})$ with $1\le i\le n$ and $0\le s\le d_i$ for the coefficients of the polynomials $Z_i$.  Write $\overline{Z_i}(u)=u^{d_i}Z_i(u^{-1})$ for the exchange polynomial with coefficients reversed.  Together we call $\Sigma=(\bfx,\bfy,B,\bfZ)$ a \emph{generalized seed over $\PP$}.  The additional data of the polynomials $\bfZ$ allows to generalize the notion of seed mutation in such a way that all nice properties and constructions related to cluster algebras in section~\ref{sec:CA} carry over to the new setting.

\begin{definition}\label{def:GCA mutation}
 For $1\le k\le n$ we define the \emph{generalized seed mutation in direction $k$} by $\mu_k(\bfx,\bfy,B,\bfZ)=(\bfx',\bfy',B',\bfZ')$ where
 \begin{itemize}
  \item the cluster $\bfx'=(x'_1,\ldots,x'_n)$ is given by $x'_i=x_i$ for $i\ne k$ and $x'_k$ is determined using the \emph{exchange relation}:
  \begin{equation}\label{eq:GCA exchange}
   x'_k x_k=\bigg(\prod_{i=1}^n x_i^{[-b_{ik}]_+}\bigg)^{d_k}\frac{Z_k\big(\hat{y}_k\big)}{Z_k\big|_\PP(y_k)},\quad \hat{y}_k=y_k\prod_{i=1}^n x_i^{b_{ik}};
  \end{equation}
  \item the coefficient tuple $\bfy'=(y'_1,\ldots,y'_n)$ is given by $y'_k=y_k^{-1}$ and for $j\ne k$ we set
  \begin{equation}\label{eq:GCA coefficient exchange}
   y'_j=y_j \big(y_k^{d_k}\big)^{[b_{kj}]_+}Z_k\big|_\PP(y_k)^{-b_{kj}};
  \end{equation}
  \item the matrix $B'=( b'_{ij})$ is given by
  \begin{equation}\label{eq:GCA matrix mutation}
   b'_{ij}=\begin{cases}
           -b_{ij} & \text{ if $i=k$ or $j=k$;}\\
           b_{ij}+[b_{ik}]_+d_k b_{kj}+b_{ik}d_k[-b_{kj}]_+ & \text{ otherwise.}\\
          \end{cases}
  \end{equation}
  \item the exchange polynomials $\bfZ'=(Z'_1,\ldots,Z'_n)$ are given by $Z'_i=Z_i$ for $i\ne k$ and $Z'_k=\overline{Z}_k$, writing this relation purely in terms of coefficients gives $z'_{i,s}=z_{i,s}$ for $i\ne k$ and $z'_{k,s}=z_{k,d_k-s}$.
 \end{itemize}
\end{definition}

One may easily check that the $\hat y$-variables mutate in the same way as the $y$-variables, namely $\hat{y}'_k=\hat{y}_k^{-1}$ and for $j\ne k$ we have
\begin{equation}\label{eq:GCA hat coefficient exchange}
 \hat{y}'_j=\hat{y}_j \big(\hat{y}_k^{d_k}\big)^{[b_{kj}]_+}Z_k(\hat{y}_k)^{-b_{kj}}.
\end{equation}

As a first indication that this definition is correct we verify that $\mu_k^2\Sigma=\Sigma$.
\begin{proposition}
 The generalized seed mutation $\mu_k$ is involutive.
\end{proposition}
\begin{proof}
 Let $(\bfx',\bfy',B',\bfZ')=\mu_k(\bfx,\bfy,B,\bfZ)$ and write $(\bfx'',\bfy'',B'',\bfZ'')=\mu_k(\bfx',\bfy',B',\bfZ')$.  To begin note that $x''_i=x'_i=x_i$ for $i\ne k$ and $\big(\hat{y}'_k\big)^{-1}=\hat{y}_k$.  Then $x''_k$ is given by
 \begin{align*}
  x_k''
  &=\frac{1}{x'_k}\bigg(\prod_{i=1}^n (x'_i)^{[-b'_{ik}]_+}\bigg)^{d_k}\frac{\overline{Z}_k\big(\hat{y}'_k\big)}{\overline{Z}_k\big|_\PP(y'_k)}=\frac{1}{x'_k}\bigg(\prod_{i=1}^n x_i^{[b_{ik}]_+}\bigg)^{d_k}\frac{\hat{y}_k^{-d_k}Z_k\big(\hat{y}_k\big)}{y_k^{-d_k}Z_k\big|_\PP(y_k)}\\
  &=\frac{1}{x'_k}\bigg(\prod_{i=1}^n x_i^{[b_{ik}]_+-b_{ik}}\bigg)^{d_k}\frac{Z_k\big(\hat{y}_k\big)}{Z_k\big|_\PP(y_k)}=\frac{1}{x'_k}\bigg(\prod_{i=1}^n x_i^{[-b_{ik}]_+}\bigg)^{d_k}\frac{Z_k\big(\hat{y}_k\big)}{Z_k\big|_\PP(y_k)}=x_k.
 \end{align*}
 Also $y_k''=(y_k')^{-1}=y_k$, while for $j\ne k$ we have
 \begin{align*}
  y_j''
  &=y'_j\big((y'_k)^{d_k}\big)^{[b'_{kj}]_+}\overline{Z}_k\big|_\PP(y'_k)^{-b'_{kj}}=y_j\big(y_k^{d_k}\big)^{[b_{kj}]_+}Z_k\big|_\PP(y_k)^{-b_{kj}} \big(y_k^{d_k}\big)^{-[-b_{kj}]_+}\Big(y_k^{-d_k}Z_k\big|_\PP(y_k)\Big)^{b_{kj}}=y_j.
 \end{align*}

 To see that the matrix mutation is involutive notice that we may apply the classical matrix mutation \eqref{eq:CA matrix mutation} to obtain exchange matrices $(DB)'$ and $(BD)'$ where $D=(d_i\delta_{ij})$.  Then it is immediate from \eqref{eq:GCA matrix mutation} that we have $DB'=(DB)'$ and $B'D=(BD)'$, the involutivity of matrix mutation \eqref{eq:GCA matrix mutation} follows.  Finally the equality $\bfZ''=\bfZ$ is immediate from the definitions.
\end{proof}

\erase{
Fix a collection $\bfd=(d_1,\ldots,d_n)$ of positive integers called \emph{mutation degrees}.  We also fix a collection $\bfz=( z_{i,s})$ where $1\le i\le n$ and $0\le s\le d_i$ called \emph{semi-frozen coefficients} or \emph{$z$-variables} subject to the requirement $ z_{i,0}= z_{i,d_i}=1$.  Together the pair $(\bfd,\bfz)$ is called \emph{mutation data}.  This data allows to generalize the notion of seed mutation in such a way that all nice properties and constructions related to cluster algebras carry over to the new setting.
\begin{definition}\label{def:GCA mutation}
 For $1\le k\le n$ we define the \emph{$(\bfd,\bfz)$-mutation in direction $k$} by $\mu_k(\bfx,\bfy, B)=(\bfx',\bfy', B')$ where
 \begin{itemize}
  \item the cluster $\bfx'=( x'_1,\ldots, x'_n)$ is given by $ x'_i= x_i$ for $i\ne k$ and $ x'_k$ is determined using the \emph{exchange relation}:
  \begin{equation}\label{eq:GCA exchange}
    x'_k x_k=\bigg(\prod_{i=1}^n  x_i^{[- b_{ik}]_+}\bigg)^{d_k}\cdot\frac{\sum\limits_{s=0}^{d_k}  z_{k,s}\hat { y}_k^s}{\bigoplus\limits_{s=0}^{d_k}  z_{k,s} y_k^s},\quad \hat { y}_k= y_k\prod_{i=1}^n  x_i^{ b_{ik}};=\frac{Z_k\Big(\prod_{i=1}^n x_i^{[-b_{ik}]_+},y_k\prod_{i=1}^n x_i^{[b_{ik}]_+}\Big)}{Z_k\big|_\PP(1,y_k)}
  \end{equation}
  \item the coefficient tuple $\bfy'=( y'_1,\ldots, y'_n)$ is given by $ y'_k= y_k^{-1}$ and for $j\ne k$ we set
  \begin{equation}\label{eq:GCA coefficient exchange}
    y'_j= y_j \Big( y_k^{[ b_{kj}]_+}\Big)^{d_k}\bigg(\bigoplus\limits_{s=0}^{d_k}  z_{k,s} y_k^s\bigg)^{- b_{kj}};
  \end{equation}
  \item the matrix $B'=( b'_{ij})$ is given by
  \begin{equation}\label{eq:GCA matrix mutation}
    b'_{ij}=\begin{cases}
           - b_{ij} & \text{ if $i=k$ or $j=k$;}\\
            b_{ij}+[ b_{ik}]_+d_k b_{kj}+ b_{ik}d_k[- b_{kj}]_+ & \text{ otherwise.}\\
          \end{cases}
  \end{equation}
 \end{itemize}
\end{definition}}

The generalized seeds and their mutations we have defined here are a specialization of the setup in \cite{chekhov-shapiro}.  There a generalized seed over $\PP$ is a triple $(\bfx,\bfp,B)$ where $\bfx$ is a cluster, $B$ is an exchange matrix, and $\bfp=(p_{i,s})$, where $1\le i\le n$ and $0\le s\le d_i$, is a collection of elements of $\PP$.  The mutation $\mu_k(\bfx,\bfp,B)=(\bfx',\bfp',B')$ is given by replacing \eqref{eq:GCA exchange} with
\begin{equation}\label{eq:CK exchange}
 x'_kx_k=\bigg(\prod_{i=1}^n x_i^{[-b_{ik}]_+}\bigg)^{d_k}\sum\limits_{s=0}^{d_k} p_{k,s} w_k^s,\quad w_k=\prod_{i=1}^n x_i^{b_{ik}}
\end{equation}
and by replacing \eqref{eq:GCA coefficient exchange} with
\[p'_{k,s}=p_{k,d_k-s}\quad\text{ and }\quad\frac{p'_{j,s}}{p'_{j,0}}=\frac{p_{j,s}}{p_{j,0}}\bigg(\frac{p_{k,d_k}}{p_{k,0}}\bigg)^{s[b_{kj}]_+}p_{k,0}^{sb_{kj}}.\]

Our generalized seed mutations can be related to the more general setting of \cite{chekhov-shapiro} by defining
\begin{equation}\label{eq:CK via NR}
 p_{i,s}=\frac{z_{i,s}y_i^s}{Z_i\big|_\PP(y_i)}
\end{equation}
where we note the identities 
\[\bigoplus_{s=0}^{d_i} p_{i,s}=1\quad\text{ and }\quad\frac{p_{i,d_i}}{p_{i,0}}=y_i^{d_i}.\]
\begin{proposition}\label{le:pca is gca}
 Generalized seeds of the form $(\bfx,\bfy,B,\bfZ)$ are in bijection with generalized seeds of the form $(\bfx,\bfp,B)$ satisfying
\begin{enumerate}
 \item (normalization condition) $\bigoplus\limits_{s=0}^{d_i} p_{i,s}=1$;
 \item (power condition) $\frac{p_{i,d_i}}{p_{i,0}}=y_i^{d_i}$ for some $y_i\in\PP$.
\end{enumerate}
Moreover, this bijection is compatible with mutations.
\end{proposition}
\begin{remark}
  Such $y_i$ as in (2) is unique since $\PP$ is torsion-free, i.e. if $(y'_i)^{d_i}=y_i^{d_i}$ then $\Big(\frac{y'_i}{y_i}\Big)^{d_i}=1$ and so $\frac{y'_i}{y_i}=1$.
 \end{remark}
\begin{proof}
 For a generalized seed $(\bfx,\bfy,B,\bfZ)$ define $p_{i,s}$ as in \eqref{eq:CK via NR}.  Write $(\bfx',\bfy',B',\bfZ')=\mu_k(\bfx,\bfy,B,\bfZ)$ and again use \eqref{eq:CK via NR} to define $p'_{i,s}$ in terms of this seed.  Then we have
 \[p'_{k,s}=\frac{z'_{k,s}(y'_k)^s}{\overline{Z}_k\big|_\PP(y'_k)}=\frac{z_{k,d_k-s}y_k^{-s}}{y_k^{-d_k}Z_k\big|_\PP(y_k)}=\frac{z_{k,d_k-s}y_k^{d_k-s}}{Z_k\big|_\PP(y_k)}=p_{k,d_k-s}\]
 while for $j\ne k$ we have
 \begin{align*}
  \frac{p'_{j,s}}{p'_{j,0}}
  &=\frac{z'_{j,s}(y'_j)^s}{\overline{Z}_j\big|_\PP(y'_j)}\frac{\overline{Z}_j\big|_\PP(y'_j)}{z'_{j,0}}=z_{j,s}\bigg(y_j\big(y_k^{d_k}\big)^{[b_{kj}]_+}Z_k\big|_\PP(y_k)^{-b_{kj}}\bigg)^s=\frac{p_{j,s}}{p_{j,0}}\bigg(\frac{p_{k,d_k}}{p_{k,0}}\bigg)^{s[b_{kj}]_+}p_{k,0}^{sb_{kj}}
 \end{align*}
 as desired.

 Conversely, let $(\bfx,\bfp,B)$ be a generalized seed satisfying (1) and (2) where we define $y_i$ using (2).  Set $z_{i,s}=y_i^{-s}\frac{p_{i,s}}{p_{i,0}}$.  Notice that the definitions immediately imply $z_{i,0}=z_{i,d_i}=1$.  Since $p_{i,s}=z_{i,s}y_i^sp_{i,0}$, by the normalization condition we have $p_{i,0}^{-1}=Z_i|_\PP(y_i)$ where we write $Z_i=z_{i,0}+z_{i,1}u+\cdots+z_{i,d_i-1}u^{d_i-1}+z_{i,d_i}u^{d_i}$.  Write $(\bfx',\bfp',B')=\mu_k(\bfx,\bfp,B)$ so that we may define $y'_i$ and $z'_{i,s}$ as above using this generalized seed.  Then we have
 \[(y'_k)^{d_k}=\frac{p'_{k,d_k}}{p'_{k,0}}=\frac{p_{k,0}}{p_{k,d_k}}=y_k^{-d_k}\]
 while for $j\ne k$ we have
 \[(y'_j)^{d_j}=\frac{p'_{j,d_j}}{p'_{j,0}}=\frac{p_{j,d_j}}{p_{j,0}}\bigg(\frac{p_{k,d_k}}{p_{k,0}}\bigg)^{d_j[ b_{kj}]_+}p_{k,0}^{d_j b_{kj}}=\bigg(y_j\big(y_k^{d_k}\big)^{[ b_{kj}]_+}Z_k\big|_\PP(y_k)^{- b_{kj}}\bigg)^{d_j},\]
 so the coefficients mutate as desired.  Similarly we have
 \begin{align*}
  z'_{k,s}=(y'_k)^{-s}\frac{p'_{k,s}}{p'_{k,0}}=y_k^s\frac{p_{k,d_k-s}}{p_{k,d_k}}=y_k^s\frac{p_{k,d_k-s}}{p_{k,0}}\frac{p_{k,0}}{p_{k,d_k}}=y_k^sy_k^{d_k-s}z_{k,d_k-s}y_k^{-d_k}=z_{k,d_k-s}
 \end{align*}
 and for $j\ne k$ we have
 \begin{align*}
  z'_{j,s}&=(y'_j)^{-s}\frac{p'_{j,s}}{p'_{j,0}}=\bigg(y_j\big(y_k^{d_k}\big)^{[ b_{kj}]_+}Z_k\big|_\PP(y_k)^{- b_{kj}}\bigg)^{-s}\frac{p_{j,s}}{p_{j,0}}\bigg(\frac{p_{k,d_k}}{p_{k,0}}\bigg)^{s[b_{kj}]_+}p_{k,0}^{sb_{kj}}=z_{j,s}
 \end{align*}
 as desired.
\end{proof}

A \emph{generalized cluster pattern $\Sigma$ over $\PP$} is in assignment of a generalized seed $\Sigma^t$ to each vertex $t\in\TT_n$ such that whenever $t\stackrel{k}{\text{---}}t'$ we have $\mu_k\Sigma^t=\Sigma^{t'}$.  As for cluster algebras, the entire generalized cluster pattern $\Sigma$ is uniquely determined from any choice of initial seed $\Sigma^{t_0}=(\bfx,\bfy, B,\bfZ)$.  We maintain the notation $\Sigma^t=(\bfx^t,\bfy^t, B^t,\bfZ^t)$ from above where we write $\bfZ^t=(Z_1^t,\ldots,Z_n^t)$.
\begin{definition}
 The \emph{generalized cluster algebra} $\cA=\cA(\bfx,\bfy, B,\bfZ)$ is the $\ZZ\PP$-subalgebra of $\cF$ generated by all cluster variables from seeds appearing in the generalized cluster pattern $\Sigma$, more precisely 
 \[\cA(\bfx,\bfy,B,\bfZ)=\ZZ\PP[ x^t_i:t\in\TT_n,1\le i\le n]\subset\cF.\]
\end{definition}

The main feature of cluster algebras to which one might attribute their ubiquity is the Laurent Phenomenon, a first indication that generalized cluster algebras will find themselves as useful is the following consequence of Proposition~\ref{le:pca is gca} and \cite[Th. 2.5]{chekhov-shapiro}.
\begin{corollary}
 Fix an initial generalized seed $(\bfx,\bfy,B,\bfZ)$ over a semifield $\PP$.  For any vertex $t\in\TT_n$ each cluster variable $x_i^t$ can be expressed as a Laurent polynomial of $\bfx$ with coefficients in $\ZZ\PP$.
\end{corollary}

\begin{example}\label{ex:type G2}
 Consider the rank 2 generalized seed $(\bfx,\bfy,B,\bfZ)$ over $\PP$ where $\bfx=(x_1,x_2)$, $\bfy=(y_1,y_2)$, $B=\left[\begin{array}{cc}0 & -1\\ 1 & 0\end{array}\right]$, and $\bfZ=(Z_1,Z_2)$ where $Z_1(u)=1+z_1u+z_2u^2+u^3$ and $Z_2(u)=1+u$.  In this case we have $\hat y_1=y_1 x_2$ and $\hat y_2=y_2 x_1^{-1}$.  Write $\Sigma(1)=(\bfx(1),\bfy(1),B(1),\bfZ(1))$ for the initial generalized seed $(\bfx,\bfy,B,\bfZ)$ and define seeds $\Sigma(t)$ for $t=2,\ldots,9$ inductively via the alternating mutation sequence below:
 \begin{equation}\label{eq:mut seq}
  \Sigma(1)\stackrel{\mu_1}{\longleftrightarrow}\Sigma(2)\stackrel{\mu_2}{\longleftrightarrow}\Sigma(3)\stackrel{\mu_1}{\longleftrightarrow}\Sigma(4)\stackrel{\mu_2}{\longleftrightarrow}\Sigma(5)\stackrel{\mu_1}{\longleftrightarrow}\Sigma(6)\stackrel{\mu_2}{\longleftrightarrow}\Sigma(7)\stackrel{\mu_1}{\longleftrightarrow}\Sigma(8)\stackrel{\mu_2}{\longleftrightarrow}\Sigma(9).
 \end{equation}
 Then the exchange matrices and exchange polynomials of these generalized seeds are given by
 \begin{equation*}
  B(t)=(-1)^{t+1}B,\quad\quad Z_2(t)=Z_2,\quad\quad\text{and}\quad\quad Z_1(t)=\begin{cases}Z_1 & \text{if $t$ is odd;}\\ \overline{Z}_1 & \text{if $t$ is even.}\end{cases}
 \end{equation*}
 The resulting cluster variables and coefficients are presented in Table 1.
\begin{table}
\begin{align*}
  &\begin{cases}
  x_1(1)=x_1\\
  x_2(1)=x_2
  \end{cases} 
  && \!\!\!\!\!\!\!\begin{cases}
  y_1(1)=y_1\\
  y_2(1)=y_2
  \end{cases}\\
  &\begin{cases}
  x_1(2)=x_1^{-1}\frac{1+z_1\hat y_1+z_2\hat y_1^2+\hat y_1^3}{1\oplus z_1y_1\oplus z_2y_1^2\oplus y_1^3}\\
  x_2(2)=x_2
  \end{cases} 
  && \!\!\!\!\!\!\!\begin{cases}
  y_1(2)=y_1^{-1}\\
  y_2(2)=y_2(1\oplus z_1y_1\oplus z_2y_1^2\oplus y_1^3)
  \end{cases}\\
  &\begin{cases}
  x_1(3)=x_1^{-1}\frac{1+z_1\hat y_1+z_2\hat y_1^2+\hat y_1^3}{1\oplus z_1y_1\oplus z_2y_1^2\oplus y_1^3}\\
  x_2(3)=x_2^{-1}\frac{1+\hat y_2+z_1\hat y_1\hat y_2+z_2\hat y_1^2\hat y_2+\hat y_1^3\hat y_2}{1\oplus y_2\oplus z_1y_1y_2\oplus z_2y_1^2y_2\oplus y_1^3y_2}
  \end{cases} 
  && \!\!\!\!\!\!\!\begin{cases}
  y_1(3)=y_1^{-1}(1\oplus y_2\oplus z_1y_1y_2\oplus z_2y_1^2y_2\oplus y_1^3y_2)\\
  y_2(3)=y_2^{-1}(1\oplus z_1y_1\oplus z_2y_1^2\oplus y_1^3)^{-1}
  \end{cases}\\
  &\begin{cases}
  x_1(4)=\mathrlap{x_1x_2^{-3}\frac{\substack{1+3\hat y_2+3\hat y_2^2+\hat y_2^3+2z_1\hat y_1\hat y_2+4z_1\hat y_1\hat y_2^2+2z_1\hat y_1\hat y_2^3+z_2\hat y_1^2\hat y_2+z_1^2\hat y_1^2\hat y_2^2+3z_2\hat y_1^2\hat y_2^2+z_1^2\hat y_1^2\hat y_2^3+2z_2\hat y_1^2\hat y_2^3\\+3\hat y_1^3\hat y_2^2+z_1z_2\hat y_1^3\hat y_2^2+2\hat y_1^3\hat y_2^3+2z_1z_2\hat y_1^3\hat y_2^3+z_1\hat y_1^4\hat y_2^2+2z_1\hat y_1^4\hat y_2^3+z_2^2\hat y_1^4\hat y_2^3+2z_2\hat y_1^5\hat y_2^3+\hat y_1^6\hat y_2^3}}
  {\substack{1\oplus 3y_2\oplus 3y_2^2\oplus y_2^3\oplus 2z_1y_1y_2\oplus 4z_1y_1y_2^2\oplus 2z_1y_1y_2^3\oplus z_2y_1^2y_2\oplus z_1^2y_1^2y_2^2\oplus 3z_2y_1^2y_2^2\oplus z_1^2y_1^2y_2^3\oplus 2z_2y_1^2y_2^3\\\oplus 3y_1^3y_2^2\oplus z_1z_2y_1^3y_2^2\oplus 2y_1^3y_2^3\oplus 2z_1z_2y_1^3y_2^3\oplus z_1y_1^4y_2^2\oplus 2z_1y_1^4y_2^3\oplus z_2^2y_1^4y_2^3\oplus 2z_2y_1^5y_2^3\oplus y_1^6y_2^3}}}\\
  x_2(4)=x_2^{-1}\frac{1+\hat y_2+z_1\hat y_1\hat y_2+z_2\hat y_1^2\hat y_2+\hat y_1^3\hat y_2}{1\oplus y_2\oplus z_1y_1y_2\oplus z_2y_1^2y_2\oplus y_1^3y_2}
  \end{cases}\\
  &&& \!\!\!\!\!\!\!\begin{cases}
  y_1(4)=y_1(1\oplus y_2\oplus z_1y_1y_2\oplus z_2y_1^2y_2\oplus y_1^3y_2)^{-1}\\
  y_2(4)=y_1^{-3}y_2^{-1}(1\oplus 3y_2\oplus 3y_2^2\oplus y_2^3\oplus 2z_1y_1y_2\oplus 4z_1y_1y_2^2\oplus 2z_1y_1y_2^3\\
  \quad\quad\quad\quad\oplus z_2y_1^2y_2\oplus z_1^2y_1^2y_2^2\oplus 3z_2y_1^2y_2^2\oplus z_1^2y_1^2y_2^3\oplus 2z_2y_1^2y_2^3\\
  \quad\quad\quad\quad\oplus z_1z_2y_1^3y_2^2\oplus 2z_1z_2y_1^3y_2^3\oplus 3y_1^3y_2^2\oplus 2y_1^3y_2^3\\
  \quad\quad\quad\quad\oplus z_1y_1^4y_2^2\oplus 2z_1y_1^4y_2^3\oplus z_2^2y_1^4y_2^3\oplus 2z_2y_1^5y_2^3\oplus y_1^6y_2^3)
  \end{cases}\\
  &\begin{cases}
  x_1(5)=\mathrlap{x_1x_2^{-3}\frac{\substack{1+3\hat y_2+3\hat y_2^2+\hat y_2^3+2z_1\hat y_1\hat y_2+4z_1\hat y_1\hat y_2^2+2z_1\hat y_1\hat y_2^3+z_2\hat y_1^2\hat y_2+z_1^2\hat y_1^2\hat y_2^2+3z_2\hat y_1^2\hat y_2^2+z_1^2\hat y_1^2\hat y_2^3+2z_2\hat y_1^2\hat y_2^3\\+3\hat y_1^3\hat y_2^2+z_1z_2\hat y_1^3\hat y_2^2+2\hat y_1^3\hat y_2^3+2z_1z_2\hat y_1^3\hat y_2^3+z_1\hat y_1^4\hat y_2^2+2z_1\hat y_1^4\hat y_2^3+z_2^2\hat y_1^4\hat y_2^3+2z_2\hat y_1^5\hat y_2^3+\hat y_1^6\hat y_2^3}}
  {\substack{1\oplus 3y_2\oplus 3y_2^2\oplus y_2^3\oplus 2z_1y_1y_2\oplus 4z_1y_1y_2^2\oplus 2z_1y_1y_2^3\oplus z_2y_1^2y_2\oplus z_1^2y_1^2y_2^2\oplus 3z_2y_1^2y_2^2\oplus z_1^2y_1^2y_2^3\oplus 2z_2y_1^2y_2^3\\\oplus 3y_1^3y_2^2\oplus z_1z_2y_1^3y_2^2\oplus 2y_1^3y_2^3\oplus 2z_1z_2y_1^3y_2^3\oplus z_1y_1^4y_2^2\oplus 2z_1y_1^4y_2^3\oplus z_2^2y_1^4y_2^3\oplus 2z_2y_1^5y_2^3\oplus y_1^6y_2^3}}}\\
  x_2(5)=\mathrlap{x_1x_2^{-2}\frac{1+2\hat y_2+\hat y_2^2+z_1\hat y_1\hat y_2+z_1\hat y_1\hat y_2^2+z_2\hat y_1^2\hat y_2^2+\hat y_1^3\hat y_2^2}{1\oplus 2y_2\oplus y_2^2\oplus z_1y_1y_2\oplus z_1y_1y_2^2\oplus z_2y_1^2y_2^2\oplus y_1^3y_2^2}}
  \end{cases}\\
  &&& \!\!\!\!\!\!\!\begin{cases}
  y_1(5)=y_1^{-2}y_2^{-1}(1\oplus 2y_2\oplus y_2^2\oplus z_1y_1y_2\oplus z_1y_1y_2^2\oplus z_2y_1^2y_2^2\oplus y_1^3y_2^2)\\
  y_2(5)=y_1^3y_2(1\oplus 3y_2\oplus 3y_2^2\oplus y_2^3\oplus 2z_1y_1y_2\oplus 4z_1y_1y_2^2\oplus 2z_1y_1y_2^3\\
  \quad\quad\quad\quad\oplus z_2y_1^2y_2\oplus z_1^2y_1^2y_2^2\oplus 3z_2y_1^2y_2^2\oplus z_1^2y_1^2y_2^3\oplus 2z_2y_1^2y_2^3\\
  \quad\quad\quad\quad\oplus 3y_1^3y_2^2\oplus z_1z_2y_1^3y_2^2\oplus 2y_1^3y_2^3\oplus 2z_1z_2y_1^3y_2^3\\
  \quad\quad\quad\quad\oplus z_1y_1^4y_2^2\oplus 2z_1y_1^4y_2^3\oplus z_2^2y_1^4y_2^3\oplus 2z_2y_1^5y_2^3\oplus y_1^6y_2^3)^{-1}
  \end{cases}\\
  &\begin{cases}
  x_1(6)=\mathrlap{x_1^2x_2^{-3}\frac{1+3\hat y_2+3\hat y_2^2+\hat y_2^3+z_1\hat y_1\hat y_2+2z_1\hat y_1\hat y_2^2+z_1\hat y_1\hat y_2^3+z_2\hat y_1^2\hat y_2^2+z_2\hat y_1^2\hat y_2^3+\hat y_1^3\hat y_2^3}{1\oplus 3y_2\oplus 3y_2^2\oplus y_2^3\oplus z_1y_1y_2\oplus 2z_1y_1y_2^2\oplus z_1y_1y_2^3\oplus z_2y_1^2y_2^2\oplus z_2y_1^2y_2^3\oplus y_1^3y_2^3}}\\
  x_2(6)=\mathrlap{x_1x_2^{-2}\frac{1+2\hat y_2+\hat y_2^2+z_1\hat y_1\hat y_2+z_1\hat y_1\hat y_2^2+z_2\hat y_1^2\hat y_2^2+\hat y_1^3\hat y_2^2}{1\oplus 2y_2\oplus y_2^2\oplus z_1y_1y_2\oplus z_1y_1y_2^2\oplus z_2y_1^2y_2^2\oplus y_1^3y_2^2}}
  \end{cases}\\
  &&& \!\!\!\!\!\!\!\begin{cases}
  y_1(6)=y_1^2y_2(1\oplus 2y_2\oplus y_2^2\oplus z_1y_1y_2\oplus z_1y_1y_2^2\oplus z_2y_1^2y_2^2\oplus y_1^3y_2^2)^{-1}\\
  y_2(6)=y_1^{-3}y_2^{-2}(1\oplus 3y_2\oplus 3y_2^2\oplus y_2^3\oplus z_1y_1y_2\oplus 2z_1y_1y_2^2\oplus z_1y_1y_2^3\\
  \quad\quad\quad\quad\oplus z_2y_1^2y_2^2\oplus z_2y_1^2y_2^3\oplus y_1^3y_2^3)
  \end{cases}\\
  &\begin{cases}
  x_1(7)=\mathrlap{x_1^2x_2^{-3}\frac{1+3\hat y_2+3\hat y_2^2+\hat y_2^3+z_1\hat y_1\hat y_2+2z_1\hat y_1\hat y_2^2+z_1\hat y_1\hat y_2^3+z_2\hat y_1^2\hat y_2^2+z_2\hat y_1^2\hat y_2^3+\hat y_1^3\hat y_2^3}{1\oplus 3y_2\oplus 3y_2^2\oplus y_2^3\oplus z_1y_1y_2\oplus 2z_1y_1y_2^2\oplus z_1y_1y_2^3\oplus z_2y_1^2y_2^2\oplus z_2y_1^2y_2^3\oplus y_1^3y_2^3}}\\
  x_2(7)=x_1x_2^{-1}\frac{1+\hat y_2}{1\oplus y_2}
  \end{cases}\\
  &&& \!\!\!\!\!\!\!\begin{cases}
  y_1(7)=y_1^{-1}y_2^{-1}(1\oplus y_2)\\
  y_2(7)=y_1^3y_2^2(1\oplus 3y_2\oplus 3y_2^2\oplus y_2^3\oplus z_1y_1y_2\oplus 2z_1y_1y_2^2\oplus z_1y_1y_2^3\\
  \quad\quad\quad\quad\oplus z_2y_1^2y_2^2\oplus z_2y_1^2y_2^3\oplus y_1^3y_2^3)^{-1}
  \end{cases}\\
  &\begin{cases}
  x_1(8)=x_1\\
  x_2(8)=x_1x_2^{-1}\frac{1+\hat y_2}{1\oplus y_2}
  \end{cases} 
  && \!\!\!\!\!\!\!\begin{cases}
  y_1(8)=&y_1y_2(1\oplus y_2)^{-1}\\
  y_2(8)=&y_2^{-1}
  \end{cases}\\
  &\begin{cases}
  x_1(9)=x_1\\
  x_2(9)=x_2
  \end{cases} 
  && \!\!\!\!\!\!\!\begin{cases}
  y_1(9)=&y_1\\
  y_2(9)=&y_2
  \end{cases}
 \end{align*}
\caption{Cluster variables and coefficients for the mutation sequence \eqref{eq:mut seq}.}
\end{table}
\end{example}

Following the same formal procedure as in section~\ref{sec:CA}, we may define \emph{$X$-functions} $X^t_i\in\QQ_\sf(\bfx,\bfy,\bfz)$ and \emph{$Y$-functions} $Y^t_j\in\QQ_\sf(\bfy,\bfz)$ by computing $x_i^t$ and $y_i^t$, respectively, in the field $\QQ(\bfx,\bfy,\bfz)$ where $\bfx$, $\bfy$, and $\bfz$ represent collections of formal indeterminates.  Using that $z_{i,0}=z_{i,d_i}=1$, the specialization of the $Y$-functions in the tropical semifield $\PP=\Trop(\bfy,\bfz)$ again produces monomials $Y^t_j\big|_{\Trop(\bfy,\bfz)}=\stackrel[i=1]{n}{\prod} y_i^{c^t_{ij}}$ where we write $C^t$ for the resulting matrix whose columns $\bfc^t_j\in\ZZ^n$ we continue to call \emph{$c$-vectors}.
\begin{proposition}(cf. \cite[Prop. 3.8]{nakanishi})\label{prop:GCA c-vectors}
 The $c$-vectors satisfy the following recurrence relation for $t\stackrel{k}{\text{---}}t'$:
 \begin{equation}
  c_{ij}^{t'}=\begin{cases}-c_{ik}^t & \text{if $j=k$;}\\ c_{ij}^t+c_{ik}^t[d_k b_{kj}^t]_++[-c_{ik}^t]_+ d_k b_{kj}^t & \text{if $j\ne k$.}\end{cases}
 \end{equation}
\end{proposition}
\begin{remark}
 It immediately follows that the $c$-vectors of $\cA(\bfx,\bfy,B,\bfZ)$ do not depend on the choice of exchange polynomials $\bfZ$, only their degrees.
\end{remark}

As in section~\ref{sec:CA} the $X$-functions become particularly nice.
\begin{proposition}(cf. \cite[Prop. 3.3]{nakanishi})\label{prop:GCA X-function polynomiality}
 Each $X$-function $X^t_i$ is contained in $\ZZ[\bfx^{\pm1},\bfy,\bfz]$.
\end{proposition}
Using essentially the same $\ZZ^n$-grading these $X$-functions will once again be homogeneous.
\begin{proposition}(cf. \cite[Prop. 3.15]{nakanishi})\label{prop:GCA g-vectors}
 Under the $\ZZ^n$-grading
 \[\deg(x_i)=\bfe_i,\quad\quad\deg(y_j)=-\bfb_j,\quad\quad\text{and}\quad\quad\deg(z_{i,s})=\boldsymbol{0},\]
 each $X$-function is homogeneous and we write $\deg\big(X^t_j\big)=\bfg^t_j=\sum\limits_{i=1}^n g^t_{ij}\bfe_i$.  Moreover, these $g$-vectors satisfy the following recurrence relation for $t\stackrel{k}{\text{---}}t'$:
 \[g^{t'}_{ij}=\begin{cases} g^t_{ij} & \text{if $j\ne k$;}\\ -g^t_{ik}+\sum\limits_{\ell=1}^n g_{i\ell}^t[-b_{\ell k}^t d_k]_+-\sum\limits_{\ell=1}^n b_{i\ell}^t[-c_{\ell k}^t d_k]_+ & \text{if $j=k$.}\end{cases}\]
\end{proposition}
\begin{remark}
 It immediately follows that the $g$-vectors of $\cA(\bfx,\bfy,B,\bfZ)$ also do not depend on the choice of exchange polynomials $\bfZ$, only their degrees.
\end{remark}

Continuing to follow the developments of section~\ref{sec:CA} we may define $F$-polynomials $ F^t_i(\bfy,\bfz)\in\ZZ[\bfy,\bfz]$ by specializing all cluster variables $ x_i$ to $1$ in the $X$-functions, i.e. $F^t_i(\bfy,\bfz)=X^t_i(\boldsymbol{1},\bfy,\bfz)$.
\begin{proposition}(cf. \cite[Prop. 3.12]{nakanishi})\label{prop:GCA F-polynomials}
 The $F$-polynomials satisfy the following recurrence relation for $t\stackrel{k}{\text{---}}t'$:
 \begin{equation}
   F_j^{t'}=\begin{cases} F_j^t & \text{if $j\ne k$;}\\ \displaystyle\big(F_k^t\big)^{-1}\bigg(\stackrel[i=1]{n}{\prod}y_i^{[-c_{ik}^t]_+}\big(F_i^t\big)^{[-b_{ik}^t]_+}\bigg)^{d_k}Z_k\bigg(\stackrel[i=1]{n}{\prod} y_i^{c_{ik}^t}\big(F_i^t\big)^{b_{ik}^t}\bigg) & \text{if $j=k$.}\end{cases}
 \end{equation}
\end{proposition}
\erase{
\begin{proposition}(cf. \cite[Prop. 3.12]{nakanishi})\label{prop:GCA F-polynomials}
 The $F$-polynomials satisfy the following recurrence relation for $t\stackrel{k}{\text{---}}t'$:
 \begin{equation}
   F_j^{t'}=\begin{cases} F_j^t & \text{if $j\ne k$;}\\ \displaystyle\big( F_k^t\big)^{-1}\bigg(\stackrel[i=1]{n}{\prod}  y_i^{[- c_{ik}^t]_+}\big( F_i^t\big)^{[- b_{ik}^t]_+}\bigg)^{d_k}\sum\limits_{s=0}^{d_k} z_{k,s}\bigg(\stackrel[i=1]{n}{\prod}  y_i^{ c_{ik}^t}\big( F_i^t\big)^{ b_{ik}^t}\bigg)^s & \text{if $j=k$.}\end{cases}
 \end{equation}
\end{proposition}}

The coefficients $ y^t_j$ can still be computed using the $c$-vectors and $F$-polynomials.
\begin{theorem}(cf. \cite[Th. 3.23]{nakanishi})\label{th:GCA y-variables}
 Fix an initial generalized seed $(\bfx,\bfy,B,\bfZ)$ over a semifield $\PP$.  For any vertex $t\in\TT_n$ each coefficient $y^t_j$ of $\cA(\bfx,\bfy,B,\bfZ)$ can be computed as
 \begin{equation}
  y_j^t=\bigg(\stackrel[i=1]{n}{\prod}y_i^{c_{ij}^t}\bigg)\stackrel[i=1]{n}{\prod} F_i^t\big|_\PP(\bfy,\bfz)^{b_{ij}^t}.
 \end{equation}
\end{theorem}

Finally the separation of additions formula still holds for cluster variables of $\cA(\bfx,\bfy,B,\bfZ)$.
\begin{theorem}(cf. \cite[Th. 3.24]{nakanishi})\label{th:GCA x-variables}
 Fix an initial generalized seed $(\bfx,\bfy,B,\bfZ)$ over a semifield $\PP$.  For any vertex $t\in\TT_n$ each cluster variable $ x^t_j$ of $\cA(\bfx,\bfy,B,\bfZ)$ can be computed as
 \begin{equation}
  x_j^t=\bigg(\stackrel[i=1]{n}{\prod} x_i^{g_{ij}^t}\bigg)\frac{F_j^t\big|_\cF(\hat{\bfy},\bfz)}{F_j^t\big|_\PP(\bfy,\bfz)},
 \end{equation}
 where $\displaystyle\hat{y}_k=y_k\prod_{i=1}^n x_i^{b_{ik}}$.
\end{theorem}

\begin{example}
 Following Theorems~\ref{th:GCA y-variables} and~\ref{th:GCA x-variables} we may immediately extract the $C$-matrix, $G$-matrix, and $F$-polynomials associated to each of the seeds $\Sigma(t)$ in Example~\ref{ex:type G2}.  Writing $C(t)$, $G(t)$, and $F(t)$ for these quantities associated to the generalized seed $\Sigma(t)$ we obtain Table 2.
 \begin{table}
 \begin{align*}
  &C(1)=\left(\!\!\begin{array}{cc}1&0\\0&1\end{array}\!\!\right), && G(1)=\left(\!\!\begin{array}{cc}1&0\\0&1\end{array}\!\!\right), && \begin{cases}F_1(1)=1\\ F_2(1)=1\end{cases}\\
  &C(2)=\left(\!\begin{array}{cc}\!\!-1&0\\0&1\end{array}\!\!\right), && G(2)=\left(\!\begin{array}{cc}\!\!-1&0\\0&1\end{array}\!\!\right), && \begin{cases}F_1(2)=1+z_1y_1+z_2y_1^2+y_1^3\\ F_2(2)=1\end{cases}\\
  &C(3)=\left(\!\begin{array}{cc}\!\!-1&0\\0&\!\!-1\end{array}\!\!\!\right), && G(3)=\left(\!\begin{array}{cc}\!\!-1&0\\0&\!\!-1\end{array}\!\!\!\right), && \begin{cases}F_1(3)=1+z_1y_1+z_2y_1^2+y_1^3\\ F_2(3)=1+y_2+z_1y_1y_2+z_2y_1^2y_2+y_1^3y_2\end{cases}\\
  &C(4)=\left(\!\!\begin{array}{cc}1&-3\\0&-1\end{array}\!\!\right), && G(4)=\left(\!\begin{array}{cc}1&0\\ \!\!-3&\!\!-1\end{array}\!\!\!\right), && \begin{cases}F_1(4)=1+3y_2+3y_2^2+y_2^3+2z_1y_1y_2+4z_1y_1y_2^2+2z_1y_1y_2^3\\
  \quad\quad\quad\quad+z_2y_1^2y_2+z_1^2y_1^2y_2^2+3z_2y_1^2y_2^2+z_1^2y_1^2y_2^3+2z_2y_1^2y_2^3\\
  \quad\quad\quad\quad+z_1z_2y_1^3y_2^2+2z_1z_2y_1^3y_2^3+3y_1^3y_2^2+2y_1^3y_2^3\\
  \quad\quad\quad\quad+z_1y_1^4y_2^2+2z_1y_1^4y_2^3+z_2^2y_1^4y_2^3+2z_2y_1^5y_2^3+y_1^6y_2^3\\ F_2(4)=1+y_2+z_1y_1y_2+z_2y_1^2y_2+y_1^3y_2\end{cases}\\
  &C(5)=\left(\!\!\!\begin{array}{cc}-2&3\\-1&1\end{array}\!\!\right), && G(5)=\left(\!\begin{array}{cc}1&1\\ \!\!-3&\!\!-2\end{array}\!\!\!\right), && \begin{cases}F_1(5)=1+3y_2+3y_2^2+y_2^3+2z_1y_1y_2+4z_1y_1y_2^2+2z_1y_1y_2^3\\
  \quad\quad\quad\quad+z_2y_1^2y_2+z_1^2y_1^2y_2^2+3z_2y_1^2y_2^2+z_1^2y_1^2y_2^3+2z_2y_1^2y_2^3\\
  \quad\quad\quad\quad+z_1z_2y_1^3y_2^2+2z_1z_2y_1^3y_2^3+3y_1^3y_2^2+2y_1^3y_2^3\\
  \quad\quad\quad\quad+z_1y_1^4y_2^2+2z_1y_1^4y_2^3+z_2^2y_1^4y_2^3+2z_2y_1^5y_2^3+y_1^6y_2^3\\ F_2(5)=1+2y_2+y_2^2+z_1y_1y_2+z_1y_1y_2^2+z_2y_1^2y_2^2+y_1^3y_2^2\end{cases}\\
  &C(6)=\left(\!\!\begin{array}{cc}2&-3\\1&-2\end{array}\!\!\right), && G(6)=\left(\!\!\begin{array}{cc}2&1\\ \!\!-3&\!\!-2\end{array}\!\!\!\right), && \begin{cases}F_1(6)=1+3y_2+3y_2^2+y_2^3+z_1y_1y_2+2z_1y_1y_2^2+z_1y_1y_2^3\\
  \quad\quad\quad\quad+z_2y_1^2y_2^2+z_2y_1^2y_2^3+y_1^3y_2^3\\ F_2(6)=1+2y_2+y_2^2+z_1y_1y_2+z_1y_1y_2^2+z_2y_1^2y_2^2+y_1^3y_2^2\end{cases}\\
  &C(7)=\left(\!\!\!\begin{array}{cc}-1&3\\-1&2\end{array}\!\!\right), && G(7)=\left(\!\!\begin{array}{cc}2&1\\ \!\!-3&\!\!-1\end{array}\!\!\!\right), && \begin{cases}F_1(7)=1+3y_2+3y_2^2+y_2^3+z_1y_1y_2+2z_1y_1y_2^2+z_1y_1y_2^3\\
  \quad\quad\quad\quad+z_2y_1^2y_2^2+z_2y_1^2y_2^3+y_1^3y_2^3\\ F_2(7)=1+y_2\end{cases}\\
  &C(8)=\left(\!\!\begin{array}{cc}1&0\\1&\!\!-1\end{array}\!\!\!\right), && G(8)=\left(\!\!\begin{array}{cc}1&1\\0&-1\end{array}\!\!\right), && \begin{cases}F_1(8)=1\\ F_2(8)=1+y_2\end{cases}\\
  &C(9)=\left(\!\!\begin{array}{cc}1&0\\0&1\end{array}\!\!\right), && G(9)=\left(\!\!\begin{array}{cc}1&0\\0&1\end{array}\!\!\right), && \begin{cases}F_1(9)=1\\ F_2(9)=1\end{cases}
 \end{align*}
 \caption{$C$-matrices, $G$-matrices, and $F$-polynomials for the mutation sequence \eqref{eq:mut seq}.}
 \end{table}
\end{example}

\section{Companion Cluster Algebras}\label{sec:companions}
Fix an initial generalized seed $(\bfx,\bfy,B,\bfZ)$ over a semifield $\PP$.  Write $D=(d_i\delta_{ij})$ where $d_i$ is the degree of the exchange polynomial $Z_i$.  

Denote by $\Lbfx:=\bfx^{1/\bfd}$ the collection $({}^L\!x_1,\cdots,{}^L\!x_n):=(x_1^{1/d_1},\ldots,x_n^{1/d_n})$ in the extension field $\QQ\PP(\bfx^{1/\bfd})$ of $\QQ\PP(\bfx)$.  For clarity we also write $\Lbfy=\bfy$, i.e. ${}^L\!y_j=y_j$.  Define the \emph{left-companion cluster algebra} $\LcA$ of $\cA$ to be $\cA(\Lbfx,\Lbfy,DB)\subset\QQ\PP(\bfx^{1/\bfd})$.  Write $(\Lbfx^t,\Lbfy^t,\LB^t)$ for the seed associated to vertex $t\in\TT_n$ in the construction of $\LcA$ and denote by ${}^L\!\bfc^t_j$, ${}^L\!\bfg^t_j$, and ${}^L\!F^t_j$ the $c$-vectors, $g$-vectors, and $F$-polynomials of $\LcA$.

Write $\Rbfx=\bfx$, i.e. ${}^R\!x_i=x_i$, and denote by $\Rbfy:=\bfy^\bfd$ the collection $({}^R\!y_1,\ldots,{}^R\!y_n)=(y_1^{d_1},\ldots,y_n^{d_n})$.  Define the \emph{right-companion cluster algebra} $\RcA$ of $\cA$ to be $\cA(\Rbfx,\Rbfy,BD)\subset\QQ\PP(\bfx)$.  Write $(\Rbfx^t,\Rbfy^t,\RB^t)$ for the seed associated to vertex $t\in\TT_n$ in the construction of $\RcA$ and denote by ${}^R\!\bfc^t_j$, ${}^R\!\bfg^t_j$, and ${}^R\!F^t_j$ the $c$-vectors, $g$-vectors, and $F$-polynomials of $\RcA$.  

We immediately obtain the following result as a consequence of Proposition~\ref{prop:CA c-vectors} and Proposition~\ref{prop:GCA c-vectors} (cf. \cite[Props. 3.9 and 3.10]{nakanishi}).
\begin{corollary}\label{cor:c-vectors}
 The $c$-vectors of the generalized cluster algebra $\cA(\bfx,\bfy,B,\bfZ)$ coincide with the $c$-vectors of its left-companion cluster algebra $\cA(\bfx^{1/\bfd},\bfy,DB)$ while the $c$-vectors of its right-companion cluster algebra $\cA(\bfx,\bfy^\bfd,BD)$ can be obtained from those of $\cA(\bfx,\bfy,B,\bfZ)$ by the transformation ${}^R\!\bfc^t_j=d_i^{-1} c^t_{ij} d_j$.
\end{corollary}
Similarly the following result is an immediate consequence of Proposition~\ref{prop:CA g-vectors} and Proposition~\ref{prop:GCA g-vectors} (cf. \cite[Props. 3.16 and 3.17]{nakanishi}).
\begin{corollary}\label{cor:g-vectors}
 The $g$-vectors of the generalized cluster algebra $\cA(\bfx,\bfy,B,\bfZ)$ coincide with the $g$-vectors of its right-companion cluster algebra $\cA(\bfx,\bfy^\bfd,BD)$ while the $g$-vectors of its left-companion cluster algebra $\cA(\bfx^{1/\bfd},\bfy,DB)$ can be obtained from those of $\cA(\bfx,\bfy,B,\bfZ)$ by the transformation ${}^L\!\bfg^t_j=d_i g^t_{ij} d_j^{-1}$.
\end{corollary}
We see from Corollary~\ref{cor:c-vectors} and Corollary~\ref{cor:g-vectors} that the $c$- and $g$-vectors of the generalized cluster algebra $\cA(\bfx,\bfy,B,\bfZ)$ are intimately related to those of its left- and right-companion cluster algebras.  The same is true for $F$-polynomials, however the precise relationship for left- and right-companions are very different.

We begin with the left-companion.  For $1\le i\le n$ and $0\le s\le d_i$ we will write $\bfz^\bin=(z^\bin_{i,s})$ where $z^{\bin}_{i,s}={d_i\choose s}\in\ZZ$.
\begin{proposition}\label{prop:left F-polynomials}
 Let $(\bfx,\bfy,B,\bfZ)$ be a generalized seed over $\PP$.  For any $t\in\TT_n$ and any $1\le j\le n$ we have the following equalities in $\QQ_\sf(\bfy)$ and $\QQ_\sf(\bfx,\bfy)$ respectively:
 \begin{equation}
  F^t_j(\bfy,\bfz^{\bin})={}^L\!F^t_j\big({}^L\!\bfy\big)^{d_j}\quad\text{ and }\quad F^t_j(\hat{\bfy},\bfz^{\bin})={}^L\!F^t_j\big({}^L\!\hat\bfy\big)^{d_j},
 \end{equation}
 where ${}^L\!y_i=y_i$ and ${}^L\!\hat{y}_i=\hat{y}_i$.
\end{proposition}
\begin{proof}
 We will proceed by induction on the distance from $t_0$ to $t$ in $\TT_n$.  To begin, note that by definition we have $\big({}^L\!x^{t_0}_j\big)^{d_j}=\big(x_j^{1/d_j}\big)^{d_j}=x_j=x^{t_0}_j$ so that $F^{t_0}_j=1={}^L\!F^{t_0}_j$, in particular $F^{t_0}_j=\big({}^L\!F^{t_0}_j\big)^{d_j}$.  Consider $t\stackrel{k}{\text{---}}t'$ with $t'$ further from $t_0$ than $t$ and suppose $F^t_j=\big({}^L\!F^t_j\big)^{d_j}$ for all $j$.  Then by Proposition~\ref{prop:CA F-polynomials} and Proposition~\ref{prop:GCA F-polynomials} we see for $j\ne k$ that $F^{t'}_j=F^t_j=\big({}^L\!F^t_j\big)^{d_j}=\big({}^L\!F^{t'}_j\big)^{d_j}$ while taking $j=k$ we have
 \begin{align*}
  {}^L\!F^{t'}_k\big({}^L\!\bfy\big)^{d_k}
  &=\Bigg(\big({}^L\!F_k^t\big)^{-1}\bigg(\stackrel[i=1]{n}{\prod} {}^L\!y_i^{[-{}^L\!c_{ik}^t]_+}\big({}^L\!F_i^t\big)^{[-d_i b_{ik}^t]_+}\bigg)\bigg(1+\stackrel[i=1]{n}{\prod} {}^L\!y_i^{{}^L\!c_{ik}^t}\big({}^L\!F_i^t\big)^{d_i b_{ik}^t}\bigg)\Bigg)^{d_k}\\
  &=\Big(\big({}^L\!F_k^t\big)^{d_k}\Big)^{-1}\bigg(\stackrel[i=1]{n}{\prod} {}^L\!y_i^{[-{}^L\!c_{ik}^t]_+}\Big(\big({}^L\!F_i^t\big)^{d_i}\Big)^{[- b_{ik}^t]_+}\bigg)^{d_k}\sum\limits_{s=0}^{d_k}{d_k\choose s}\bigg(\stackrel[i=1]{n}{\prod} {}^L\!y_i^{{}^L\!c_{ik}^t}\Big(\big({}^L\!F_i^t\big)^{d_i}\Big)^{ b_{ik}^t}\bigg)^s\\
  &=\big( F_k^t\big)^{-1}\bigg(\stackrel[i=1]{n}{\prod} y_i^{[-c_{ik}^t]_+}\big(F_i^t\big)^{[-b_{ik}^t]_+}\bigg)^{d_k}\sum\limits_{s=0}^{d_k}{d_k\choose s}\bigg(\stackrel[i=1]{n}{\prod} y_i^{c_{ik}^t}\big(F_i^t\big)^{b_{ik}^t}\bigg)^s\\
  &= F^{t'}_k(\bfy,\bfz^{\bin}),
 \end{align*}
 where we used Corollary~\ref{cor:c-vectors} in the third equality above.  It follows by induction that $F^t_j(\bfy,\bfz^{\bin})={}^L\!F^t_j\big({}^L\!\bfy\big)^{d_j}$ for all $t\in\TT_n$ and $1\le j\le n$.  Note that $\hat{y}_j=y_j\stackrel[i=1]{n}{\prod} x_i^{b_{ij}}={}^L\!y_j\stackrel[i=1]{n}{\prod} {}^L\!x_i^{d_i b_{ij}}={}^L\!\hat y_j$ so that substituting the variables $\hat{y}_j$ into this identity gives $F^t_j(\hat{\bfy},\bfz^{\bin})={}^L\!F^t_j\big({}^L\!\hat\bfy\big)^{d_j}$ for all $t\in\TT_n$ and $1\le j\le n$.
\end{proof}
Write $x_i^t\big|_{\bfz=\bfz^\bin}\in\cF$ and $y_j^t\big|_{\bfz=\bfz^\bin}\in\PP$ for the variables obtained by applying equations \eqref{th:GCA x-variables} and \eqref{th:GCA y-variables} respectively using the specialized $F$-polynomials $F^t_j(\bfy,\bfz^{\bin})$ in place of the generic $F$-polynomials $F^t_j(\bfy,\bfz)$.
\begin{theorem}\label{th:left companions}
 We have $x_i^t\big|_{\bfz=\bfz^\bin}=\big({}^L\!x^t_i\big)^{d_i}$ and $y_j^t\big|_{\bfz=\bfz^\bin}={}^L\!y^t_i$.
\end{theorem}
\begin{proof}
 For coefficients we apply Theorem~\ref{th:CA y-variables} and Theorem~\ref{th:GCA y-variables} along with Corollary~\ref{cor:c-vectors} to get
 \begin{align*}
  {}^L\!y^t_j
  &=\bigg(\stackrel[i=1]{n}{\prod} {}^L\!y_i^{{}^L\!c^t_{ij}}\bigg)\stackrel[i=1]{n}{\prod}{}^L\!F^t_i\big|_\PP\big(\Lbfy\big)^{d_i b^t_{ij}}=\bigg(\stackrel[i=1]{n}{\prod}  y_i^{c^t_{ij}}\bigg)\stackrel[i=1]{n}{\prod} F^t_i\big|_\PP(\bfy,\bfz^{\bin})^{ b^t_{ij}}= y^t_j.
 \end{align*}

 To finish, we may apply Theorem~\ref{th:CA x-variables} and Theorem~\ref{th:GCA x-variables} along with Corollary~\ref{cor:g-vectors} to get
 \begin{align*}
  \big({}^L\!x^t_j\big)^{d_j}
  &=\Bigg(\bigg(\stackrel[i=1]{n}{\prod}{}^L\!x_i^{{}^L\!g_{ij}^t}\bigg)\frac{{}^L\!F_j^t\big|_\cF\big({}^L\!\hat\bfy\big)}{{}^L\!F_j^t\big|_\PP\big({}^L\!\bfy\big)}\Bigg)^{d_j}=\bigg(\stackrel[i=1]{n}{\prod}{}^L\!x_i^{{}^L\!g_{ij}^td_j}\bigg)\frac{{}^L\!F_j^t\big|_\cF\big({}^L\!\hat\bfy\big)^{d_j}}{{}^L\!F_j^t\big|_\PP\big({}^L\!\bfy\big)^{d_j}}\\
  &=\bigg(\stackrel[i=1]{n}{\prod}\Big({}^L\!x_i^{1/d_i}\Big)^{ g_{ij}^t}\bigg)\frac{ F_j^t\big|_\cF(\hat{\bfy},\bfz^{\bin})}{ F_j^t\big|_\PP(\bfy,\bfz^{\bin})}=\bigg(\stackrel[i=1]{n}{\prod}x_i^{ g_{ij}^t}\bigg)\frac{ F_j^t\big|_\cF(\hat{\bfy},\bfz^{\bin})}{ F_j^t\big|_\PP(\bfy,\bfz^{\bin})}= x^t_j.
 \end{align*}
\end{proof}
\begin{example}
 As an illustration of Corollaries~\ref{cor:c-vectors} and~\ref{cor:g-vectors} as well as Theorem~\ref{th:left companions} we now present the $C$-matrices, $G$-matrices, and $F$-polynomials for the left companion cluster algebra $\LcA$ in Table 3 from which we invite the reader to directly verify these results.
 \begin{table}
 \begin{align*}
  &{}^L\!C(1)=\left(\!\!\begin{array}{cc}1&0\\0&1\end{array}\!\!\right), && {}^L\!G(1)=\left(\!\!\begin{array}{cc}1&0\\0&1\end{array}\!\!\right), && \begin{cases}{}^L\!F_1(1)=1\\ {}^L\!F_2(1)=1\end{cases}\\
  &{}^L\!C(2)=\left(\!\begin{array}{cc}\!\!-1&0\\0&1\end{array}\!\!\right), && {}^L\!G(2)=\left(\!\begin{array}{cc}\!\!-1&0\\0&1\end{array}\!\!\right), && \begin{cases}{}^L\!F_1(2)=1+{}^L\!y_1\\ {}^L\!F_2(2)=1\end{cases}\\
  &{}^L\!C(3)=\left(\!\begin{array}{cc}\!\!-1&0\\0&\!\!-1\end{array}\!\!\!\right), && {}^L\!G(3)=\left(\!\begin{array}{cc}\!\!-1&0\\0&\!\!-1\end{array}\!\!\!\right), && \begin{cases}{}^L\!F_1(3)=1+{}^L\!y_1\\ {}^L\!F_2(3)=1+{}^L\!y_2+3{}^L\!y_1{}^L\!y_2+3{}^L\!y_1^2{}^L\!y_2+{}^L\!y_1^3{}^L\!y_2\end{cases}\\
  &{}^L\!C(4)=\left(\!\!\begin{array}{cc}1&-3\\0&-1\end{array}\!\!\right), && {}^L\!G(4)=\left(\!\begin{array}{cc}1&0\\ \!\!-1&\!\!-1\end{array}\!\!\!\right), && \begin{cases}{}^L\!F_1(4)=1+{}^L\!y_2+2{}^L\!y_1{}^L\!y_2+{}^L\!y_1^2{}^L\!y_2\\ 
  {}^L\!F_2(4)=1+{}^L\!y_2+3{}^L\!y_1{}^L\!y_2+3{}^L\!y_1^2{}^L\!y_2+{}^L\!y_1^3{}^L\!y_2\end{cases}\\
  &{}^L\!C(5)=\left(\!\!\!\begin{array}{cc}-2&3\\-1&1\end{array}\!\!\right), && {}^L\!G(5)=\left(\!\begin{array}{cc}1&3\\ \!\!-1&\!\!-2\end{array}\!\!\!\right), && \begin{cases}{}^L\!F_1(5)=1+{}^L\!y_2+2{}^L\!y_1{}^L\!y_2+{}^L\!y_1^2{}^L\!y_2\\ 
  {}^L\!F_2(5)=1+2{}^L\!y_2+{}^L\!y_2^2+3{}^L\!y_1{}^L\!y_2\\
  \quad\quad\quad\quad+3{}^L\!y_1{}^L\!y_2^2+3{}^L\!y_1^2{}^L\!y_2^2+{}^L\!y_1^3{}^L\!y_2^2\end{cases}\\
  &{}^L\!C(6)=\left(\!\!\begin{array}{cc}2&-3\\1&-2\end{array}\!\!\right), && {}^L\!G(6)=\left(\!\!\begin{array}{cc}2&3\\ \!\!-1&\!\!-2\end{array}\!\!\!\right), && \begin{cases}{}^L\!F_1(6)=1+{}^L\!y_2+{}^L\!y_1{}^L\!y_2\\ {}^L\!F_2(6)=1+2{}^L\!y_2+{}^L\!y_2^2+3{}^L\!y_1{}^L\!y_2\\
  \quad\quad\quad\quad+3{}^L\!y_1{}^L\!y_2^2+3{}^L\!y_1^2{}^L\!y_2^2+{}^L\!y_1^3{}^L\!y_2^2\end{cases}\\
  &{}^L\!C(7)=\left(\!\!\!\begin{array}{cc}-1&3\\-1&2\end{array}\!\!\right), && {}^L\!G(7)=\left(\!\!\begin{array}{cc}2&3\\ \!\!-1&\!\!-1\end{array}\!\!\!\right), && \begin{cases}{}^L\!F_1(7)=1+{}^L\!y_2+{}^L\!y_1{}^L\!y_2\\ {}^L\!F_2(7)=1+{}^L\!y_2\end{cases}\\
  &{}^L\!C(8)=\left(\!\!\begin{array}{cc}1&0\\1&\!\!-1\end{array}\!\!\!\right), && {}^L\!G(8)=\left(\!\!\begin{array}{cc}1&3\\0&-1\end{array}\!\!\right), && \begin{cases}{}^L\!F_1(8)=1\\ {}^L\!F_2(8)=1+{}^L\!y_2\end{cases}\\
  &{}^L\!C(9)=\left(\!\!\begin{array}{cc}1&0\\0&1\end{array}\!\!\right), && {}^L\!G(9)=\left(\!\!\begin{array}{cc}1&0\\0&1\end{array}\!\!\right), && \begin{cases}{}^L\!F_1(9)=1\\ {}^L\!F_2(9)=1\end{cases}
 \end{align*}
 \caption{$C$-matrices, $G$-matrices, and $F$-polynomials for the same mutation sequence \eqref{eq:mut seq} applied to the seeds of $\LcA$.}
 \end{table}
\end{example}

To state a relationship between a generalized cluster algebra and its right-companion we need the following analogue of Proposition~\ref{prop:left F-polynomials}.
\begin{proposition}\label{prop:right F-polynomials}
 Let $(\bfx,\bfy,B,\bfZ)$ be a generalized seed over $\PP$.  For any $t\in\TT_n$ and any $1\le j\le n$ we have the following equalities in $\QQ_\sf(\bfy)$ and $\QQ_\sf(\bfx,\bfy)$ respectively:
 \begin{equation}
  F^t_j(\bfy,\boldsymbol{0})={}^R\!F^t_j\big({}^R\!\bfy\big)\quad\text{ and }\quad F^t_j(\hat{\bfy},\boldsymbol{0})={}^R\!F^t_j\big({}^R\!\hat\bfy\big),
 \end{equation}
 where ${}^R\!y_i=y_i^{d_i}$ and ${}^R\hat{y}_i=\hat{y}_i^{d_i}$.
\end{proposition}
\begin{proof}
 We will proceed by induction on the distance from $t_0$ to $t$ in $\TT_n$.  To begin, note that by definition we have $ F^{t_0}_j=1={}^R\!F^{t_0}_j$.  Consider $t\stackrel{k}{\text{---}}t'$ with $t'$ further from $t_0$ than $t$ and suppose $ F^t_j={}^R\!F^t_j$ for all $j$.  Then by Proposition~\ref{prop:CA F-polynomials} and Proposition~\ref{prop:GCA F-polynomials} we see for $j\ne k$ that $ F^{t'}_j= F^t_j={}^R\!F^t_j={}^R\!F^{t'}_j$ while taking $j=k$ we have
 \begin{align*}
  {}^R\!F^{t'}_k\big({}^R\bfy\big)
  &=\big({}^R\!F_k^t\big)^{-1}\bigg(\stackrel[i=1]{n}{\prod} {}^R\!y_i^{[-{}^R\!c_{ik}^t]_+}\big({}^R\!F_i^t\big)^{[- b_{ik}^td_k]_+}\bigg)\bigg(1+\stackrel[i=1]{n}{\prod} {}^R\!y_i^{{}^R\!c_{ik}^t}\big({}^R\!F_i^t\big)^{ b_{ik}^td_k}\bigg)\\
  &=\big( F_k^t\big)^{-1}\bigg(\stackrel[i=1]{n}{\prod}  y_i^{[- c_{ik}^td_k]_+}\big( F_i^t\big)^{[- b_{ik}^td_k]_+}\bigg)\bigg(1+\stackrel[i=1]{n}{\prod}  y_i^{ c_{ik}^td_k}\big( F_i^t\big)^{ b_{ik}^td_k}\bigg)\\
  &=\big( F_k^t\big)^{-1}\bigg(\stackrel[i=1]{n}{\prod}  y_i^{[- c_{ik}^t]_+}\big( F_i^t\big)^{[- b_{ik}^t]_+}\bigg)^{d_k}\Bigg(1+\bigg(\stackrel[i=1]{n}{\prod}  y_i^{ c_{ik}^t}\big( F_i^t\big)^{ b_{ik}^t}\bigg)^{d_k}\Bigg)\\
  &= F^{t'}_k(\bfy,\boldsymbol{0}).
 \end{align*}
 It follows by induction that $F^t_j(\bfy,\boldsymbol{0})={}^R\!F^t_j\big({}^R\!\bfy\big)$ for all $t\in\TT_n$ and $1\le j\le n$.  Finally notice that ${}^R\!\hat y_j={}^R\!y_j\stackrel[i=1]{n}{\prod} {}^R\!x_i^{ b_{ij}d_j}=y_j^{d_j}\stackrel[i=1]{n}{\prod}  x_i^{ b_{ij}d_j}=\hat{ y}_j^{d_j}$ so that substituting the variables $\hat{y}_j$ into this identity gives $F^t_j(\hat{\bfy},\boldsymbol{0})={}^R\!F^t_j\big({}^R\!\hat\bfy\big)$ for all $t\in\TT_n$ and $1\le j\le n$.
\end{proof}

Write $x_i^t\big|_{\bfz=\boldsymbol{0}}\in\cF$ and $y_j^t\big|_{\bfz=\boldsymbol{0}}\in\PP$ for the variables obtained by applying equations \eqref{th:GCA x-variables} and \eqref{th:GCA y-variables} respectively using the specialized $F$-polynomials $F^t_j(\bfy,\boldsymbol{0})$ in place of the generic $F$-polynomials $F^t_j(\bfy,\bfz)$.
\begin{theorem}\label{th:right companions}
 We have $x_i^t\big|_{\bfz=\boldsymbol{0}}={}^Rx^t_i$ and $\big(y_j^t\big|_{\bfz=\boldsymbol{0}}\big)^{d_j}={}^Ry^t_j$.
\end{theorem}
\begin{proof}
 To see the claim for coefficients we apply Theorem~\ref{th:CA y-variables} and Theorem~\ref{th:GCA y-variables} along with Corollary~\ref{cor:c-vectors} and Proposition~\ref{prop:right F-polynomials} to get
 \begin{align*}
  {}^R\!y^t_j
  &=\bigg(\stackrel[i=1]{n}{\prod} {}^R\!y_i^{{}^R\!c^t_{ij}}\bigg)\stackrel[i=1]{n}{\prod}{}^R\!F^t_i\big|_\PP(\Rbfy)^{b^t_{ij}d_j}=\bigg(\stackrel[i=1]{n}{\prod}  y_i^{ c^t_{ij}d_j}\bigg)\stackrel[i=1]{n}{\prod}F^t_i\big|_\PP(\bfy,\boldsymbol{0})^{ b^t_{ij}d_j}=\big( y^t_j\big)^{d_j}.
 \end{align*}
 Finally to see the claim for cluster variables we apply Theorem~\ref{th:CA x-variables} and Theorem~\ref{th:GCA x-variables} along with Corollary~\ref{cor:g-vectors} and Proposition~\ref{prop:right F-polynomials} to get
 \begin{align*}
  {}^R\!x^t_j
  &=\bigg(\stackrel[i=1]{n}{\prod}{}^R\!x_i^{{}^R\!g_{ij}^t}\bigg)\frac{{}^R\!F_j^t\big|_\cF({}^R\!\hat\bfy)}{{}^R\!F_j^t\big|_\PP({}^R\!\bfy)}=\bigg(\stackrel[i=1]{n}{\prod}x_i^{ g_{ij}^t}\bigg)\frac{ F_j^t\big|_\cF(\hat{\bfy},\boldsymbol{0})}{ F_j^t\big|_\PP(\bfy,\boldsymbol{0})}= x^t_j.
 \end{align*}
\end{proof}

\begin{example}
 As an illustration of Corollaries~\ref{cor:c-vectors} and~\ref{cor:g-vectors} as well as Theorem~\ref{th:right companions} we now present the $C$-matrices, $G$-matrices, and $F$-polynomials for the right companion cluster algebra $\RcA$ in Table 4 from which we invite the reader to directly verify these results.
 \begin{table}
 \begin{align*}
  &{}^R\!C(1)=\left(\!\!\begin{array}{cc}1&0\\0&1\end{array}\!\!\right), && {}^R\!G(1)=\left(\!\!\begin{array}{cc}1&0\\0&1\end{array}\!\!\right), && \begin{cases}{}^R\!F_1(1)=1\\ {}^R\!F_2(1)=1\end{cases}\\
  &{}^R\!C(2)=\left(\!\begin{array}{cc}\!\!-1&0\\0&1\end{array}\!\!\right), && {}^R\!G(2)=\left(\!\begin{array}{cc}\!\!-1&0\\0&1\end{array}\!\!\right), && \begin{cases}{}^R\!F_1(2)=1+{}^R\!y_1\\ {}^R\!F_2(2)=1\end{cases}\\
  &{}^R\!C(3)=\left(\!\begin{array}{cc}\!\!-1&0\\0&\!\!-1\end{array}\!\!\!\right), && {}^R\!G(3)=\left(\!\begin{array}{cc}\!\!-1&0\\0&\!\!-1\end{array}\!\!\!\right), && \begin{cases}{}^R\!F_1(3)=1+{}^R\!y_1\\ {}^R\!F_2(3)=1+{}^R\!y_2+{}^R\!y_1{}^R\!y_2\end{cases}\\
  &{}^R\!C(4)=\left(\!\!\begin{array}{cc}1&-1\\0&-1\end{array}\!\!\right), && {}^R\!G(4)=\left(\!\begin{array}{cc}1&0\\ \!\!-3&\!\!-1\end{array}\!\!\!\right), && \begin{cases}{}^R\!F_1(4)=1+3{}^R\!y_2+3{}^R\!y_2^2+{}^R\!y_2^3\\
  \quad\quad\quad\quad+3{}^R\!y_1{}^R\!y_2^2+2{}^R\!y_1{}^R\!y_2^3+{}^R\!y_1^2{}^R\!y_2^3\\ {}^R\!F_2(4)=1+{}^R\!y_2+{}^R\!y_1{}^R\!y_2\end{cases}\\
  &{}^R\!C(5)=\left(\!\!\!\begin{array}{cc}-2&1\\-3&1\end{array}\!\!\right), && {}^R\!G(5)=\left(\!\begin{array}{cc}1&1\\ \!\!-3&\!\!-2\end{array}\!\!\!\right), && \begin{cases}{}^R\!F_1(5)=1+3{}^R\!y_2+3{}^R\!y_2^2+{}^R\!y_2^3\\
  \quad\quad\quad\quad+3{}^R\!y_1{}^R\!y_2^2+2{}^R\!y_1{}^R\!y_2^3+{}^R\!y_1^2{}^R\!y_2^3\\ {}^R\!F_2(5)=1+2{}^R\!y_2+{}^R\!y_2^2+{}^R\!y_1{}^R\!y_2^2\end{cases}\\
  &{}^R\!C(6)=\left(\!\!\begin{array}{cc}2&-1\\3&-2\end{array}\!\!\right), && {}^R\!G(6)=\left(\!\!\begin{array}{cc}2&1\\ \!\!-3&\!\!-2\end{array}\!\!\!\right), && \begin{cases}{}^R\!F_1(6)=1+3{}^R\!y_2+3{}^R\!y_2^2+{}^R\!y_2^3+{}^R\!y_1{}^R\!y_2^3\\ {}^R\!F_2(6)=1+2{}^R\!y_2+{}^R\!y_2^2+{}^R\!y_1{}^R\!y_2^2\end{cases}\\
  &{}^R\!C(7)=\left(\!\!\!\begin{array}{cc}-1&1\\-3&2\end{array}\!\!\right), && {}^R\!G(7)=\left(\!\!\begin{array}{cc}2&1\\ \!\!-3&\!\!-1\end{array}\!\!\!\right), && \begin{cases}{}^R\!F_1(7)=1+3{}^R\!y_2+3{}^R\!y_2^2+{}^R\!y_2^3+{}^R\!y_1{}^R\!y_2^3\\ {}^R\!F_2(7)=1+{}^R\!y_2\end{cases}\\
  &{}^R\!C(8)=\left(\!\!\begin{array}{cc}1&0\\3&\!\!-1\end{array}\!\!\!\right), && {}^R\!G(8)=\left(\!\!\begin{array}{cc}1&1\\0&-1\end{array}\!\!\right), && \begin{cases}{}^R\!F_1(8)=1\\ {}^R\!F_2(8)=1+{}^R\!y_2\end{cases}\\
  &{}^R\!C(9)=\left(\!\!\begin{array}{cc}1&0\\0&1\end{array}\!\!\right), && {}^R\!G(9)=\left(\!\!\begin{array}{cc}1&0\\0&1\end{array}\!\!\right), && \begin{cases}{}^R\!F_1(9)=1\\ {}^R\!F_2(9)=1\end{cases}
 \end{align*}
 \caption{$C$-matrices, $G$-matrices, and $F$-polynomials for the same mutation sequence \eqref{eq:mut seq} applied to the seeds of $\RcA$.}
 \end{table}
\end{example}

\end{document}